\documentclass[10pt]{amsart}
\usepackage{top_horrocks_ARXIV_2}

\begin{document}
\title[Topological Horrocks addition]{Rank-preserving additions for topological vector bundles, after a construction of Horrocks}
\author{Morgan Opie}

\maketitle
\begin{abstract}
We produce group structures on certain sets of topological vector bundles of fixed rank. In particular, we put a group structure on complex rank $2$ bundles on $\mathbb{C}P^3$ with fixed first Chern class. We show that this binary operation coincides with a construction on locally free sheaves due to Horrocks, provided Horrocks' construction is defined. Using similar ideas, we give a group structures on certain sets of rank $3$ bundles on $\mathbb{C}P^5$.

These groups arise from the study of relative infinite loop space structures on truncated diagrams. Specifically, we show that the $(2n-2)$-truncation of an $n$-connective map $X\to Y$ with a section is a highly structured group object over the $(2n-2)$-truncation of $Y$. Applying these results to classifying spaces yields the group structures of interest.
\end{abstract}
\setcounter{tocdepth}{2}
\tableofcontents

\section{Introduction}\label{sec:intro}

Despite the importance of vector bundles in geometry and topology, there are few explicit methods to produce them. On complex projective spaces, the simplest complex bundles to write down are sums of line bundles. Indecomposable bundles are difficult to describe explicitly, but there are some famous examples: one is that of the Horrocks--Mumford bundle of rank $2$ on $\CP^4$ \cite{HorMum}; another is the Horrocks bundles of rank $3$ on $\CP^5$ \cite{Hor2}. 

In \cite{Hor}, Horrocks takes another approach and constructs new algebraic vector 
bundles from given ones using a modified extension group procedure. 
Horrocks' construction takes as input
 complex rank two algebraic bundles 
$V$ and $W$ on $\CP^3$, which must have the same first Chern class and 
satisfy some technical hypotheses, and outputs another complex rank $2$ algebraic vector 
bundle with the same first Chern class. We will write $V+_H W$ for the output bundle, 
although the construction does not define a group structure.

Atiyah and Rees show that Horrocks' construction produces essentially all topological equivalence classes of complex rank $2$ bundles on $\CP^3$ from the simplest ones.
\begin{thm}[{\cite[Theorem 1.1]{AR}}]\label{thm:horrocks_generates} Any complex rank $2$ topological vector bundle on $\CP^3$ can be obtained from a sum of line bundles by iteratively applying the following operations:
\begin{itemize}
\item tensoring by a line bundle; and
\item applying Horrocks' construction.
\end{itemize}
\end{thm}
\begin{rmk} This result shows that every topological equivalence class of complex rank $2$ vector bundles on $\CP^3$ admits an algebraic representative. \end{rmk}
The algebra involved in Horrocks' construction is quite specialized and does not directly generalize to produce algebraic vector bundles of other ranks or on other spaces. 
However, we explore Horrocks' construction from a homotopical vantage point and show that, from this perspective, it does generalize.

\begin{prop}[Topological Horrocks addition]\label{prop:main_construction}
Fix an integer $a_1\in \mathbb Z$. Let $\G_{a_1}$ denote the set of topological isomorphism classes of complex rank $2$ topological bundles on $\CP^3$ with first Chern class equal to $a_1$.
\begin{enumerate}
\item[(i)] $\G_{a_1}$ carries an abelian group structure $+_{a_1}$, via an explicit construction on classifying spaces;
\item[(ii)] The identity is $L\oplus \underline{\C}$, where $L$ is the complex line bundle determined by $c_1(L)=a_1$ and $\underline \C$ is the trivial complex line bundle; and
\item[(iii)] The second Chern class defines a homomorphism $c_2\: \G_{a_1} \to H^4(\CP^3;\Z).$
\end{enumerate}
\end{prop}
\begin{rmk} The existence of a group structure on the underlying set of $\mathcal G_{a_1}$ can also be deduced by identifying elements of $\mathcal G_{a_1}$ with twisted symplectic $K$-theory classes. It is not hard to show that this group structure agrees with ours, but our description provides a framework for comparison with Horrocks' construction and for generalization.\end{rmk}
We show $+_{a_1}$ addition agrees with Horrocks' construction as far as possible.
\begin{thm}\label{cor:alg_top_agree_somewhat}Suppose that $V$ and $W$ are two rank $2$ algebraic vector bundles such that $c_1(V)=c_1(W)=a_1$ and such that the Horrocks sum $V+_HW$ is defined. Then
$$V+_H W \simeq V +_{a_1} W$$ as topological vector bundles.
\end{thm}

The proof of \Cref{cor:alg_top_agree_somewhat} is conceptually simple. In \cite{AR}, Atiyah and Rees show that a topological rank $2$ bundle on $\CP^3$ is determined by its first and second Chern class together with a $\Z/2$-valued invariant called $\alpha$.
So, to prove that $V+_H W$ and $V+_{a_1} W$ are topologically isomorphic, it suffices to show their Chern classes and $\alpha$-invariants agree.

The groups of \Cref{prop:main_construction} generalize as follows.
\begin{thm}[Topological Horrocks addition for rank $3$ bundles]\label{prop:main_construction2}
Fix a complex rank $2$ topological bundle $V_0$ on $\CP^5$. Let $\G_{V_0}$ denote the set of topological equivalence classes of complex rank $3$ topological bundles on $\CP^5$ with the same first and second
Chern class as $V_0$.
\begin{enumerate}
\item[(i)] $\G_{V_0}$ carries an abelian group structure, via an explicit construction on classifying spaces;
\item[(ii)] The identity is $V_0 \oplus \underline{\C};$ and
\item[(iii)] The third Chern class defines a homomorphism $c_3\: {\G}_{V_0} \to H^6(\CP^5;\Z).$
\end{enumerate}
\end{thm}
\begin{rmk} The underlying set of $\G_{V_0}$ depends only on a choice of first and second Chern class of $V_0$. However, we need an auxiliary rank $2$ bundle to define the group operation. \end{rmk}

\subsecl{Methods}{subsec:tools}

\Cref{prop:main_construction} and \Cref{prop:main_construction2} follow from a more general study of relative infinite loop space structures on truncated diagrams. We remind the reader of the following terminology.
\begin{defn}\label{defn:basic}  Let $\Top$ denote the category of topological spaces. 
\begin{enumerate}
\item[a.] The $m$-truncation functor $\tau_{m}\: \Top \to \Top_{\leq m}$ is reflection onto the full subcategory $\Top_{\leq m}$ of spaces with $X$ with $\pi_iX=0$ for $i>m$. The functor $\tau_m$ is left adjoint to the inclusion $\Top_{\leq m} \hookrightarrow \Top.$
\item[b.] A map $f\: X \to Y$ in $\Top$ is $n$-connective if the homotopy fiber of $f$ has first nonzero homotopy group in degree $n$. 
\item[c.] In any category, given objects $X$ and $Y$ and a morphism $s\: Y \to X$, let $X \cup_Y \!X$ denote the pushout of the diagram
\[\begin{tikzcd}[row sep=.2in, column sep=.2in] Y \ar[r,"s"] \ar[d,"s"]& X \\
X & 
\end{tikzcd}
\] 
Let $1\: X \to X$ denote the identity. The map $\fold\: X\cup_YX \to X$ is the dotted arrow making the following diagram commute:
\[\begin{tikzcd}[row sep=.18in, column sep=.2in] Y \ar[r,"s"] \ar[d,"s"]& X\ar[d]\ar[ddr,bend left,"1"]\\
X\ar[r]\ar[rrd,bend right, "1"] & X \cup_Y\!\! X\ar[dr,dashed,"\fold"] \\ 
& &  X 
\end{tikzcd}
\] 
\end{enumerate}
\end{defn}
In the homotopy category of spaces, our key technical result can be stated as follows.

\begin{prop}\label{prop:maps_group}
Let $f\: X \to Y$ be a map of pointed, simply connected spaces, 
with section $s\: Y \to X$. Suppose that $f$ is $n$-connective and let $m:=2n-2.$ Let $\taum$ denote the $m$-truncation functor.
Then \[\begin{tikzcd} \tau_m X \arrow[r," \tau_m f"] & \tau_m Y
\arrow[l,dashed, bend left, " \tau_m s"] \end{tikzcd}\]
is a group object in the homotopy category of $m$-truncated spaces over $\taum Y,$ with binary operation
given by truncating the diagram

\begin{center}\begin{tikzcd}[column sep=5.5em] X \times_Y X \arrow[r, dashed, bend right] & X \cup_Y X \arrow[l, "(1\cup sf) \times (sf \cup 1)"] \arrow[r, "\operatorname{fold}"] & X
\etikz where the dotted arrow exists only after truncating.
\end{prop}
Using the skeleton-truncation adjunction on the homotopy category of spaces, we obtain:

\begin{cor}\label{cor:add_pi_zero}
Let $f\: X \to Y$, $s\: Y \to X$, $n$, and $m$ be as in \Cref{prop:maps_group}. Let $C$ be an $m$-skeletal space and fix $c\: C \to Y.$ Let $[C,X]_c$ denote homotopy classes of maps $g\:C\to X$
\begin{equation}\label{eq:diagram7}
\begin{tikzcd}
C \ar[r, dashed, "g"]\ar[dr,"c" below] & X\ar[d,"f" left]\\
& Y
\end{tikzcd}
\end{equation}
making Diagram~\eqref{eq:diagram7} homotopy commutative. Then $[C,X]_c$ is a group with identity $s\circ c$.
\end{cor}
\begin{rmk} The previous corollary can be viewed as a relative version of Borsuk's classical construction, as given in \cite{Borsuk_Spanier}. More recently, such ideas have been used in \cite{AF18}, for example. \end{rmk}

It is rather tedious to prove \Cref{prop:maps_group} in a point-set way. Instead, we prove an $\infone$-categorical refinement. As a consequence of $n$-connectivity, the fibers of $\tau_m f\:\tau_m X \to \tau_m Y$ are infinite loop spaces and the automorphisms of the fibers induced by paths in the base are infinite loop maps; this data should assemble to make $\tau_m X \to \tau_mY$ an infinite loop space over $\tau_m Y$. The machinery of an $\infone$-categorical Grothendieck construction, for example straightening and unstraightening as in \cite{HTT}, allows us to make this precise.
\begin{prop}\label{prop:general_add}Let $f\: X \to Y$ be a map of pointed simply connected spaces with section $s\: Y \to X$. Suppose $f$ has homotopy fiber with first nonzero homotopy group in degree $n$ and let $m:=2n-2.$
The functor of infinity categories $$\Groth_*(\taum f)\: \tau_mY \to \s_*,$$ given by applying pointed straightening $\Groth_*$ to the $m$-truncation of $f$, factors as
\begin{center}
\begin{tikzcd}
\tau_mY \arrow[rr,"{\Groth_*(\taum f)}"]\arrow[d, dashed, "\exists" left]& &\s_* & \\
\sp \arrow[urr, "\,\,\,\, \Omega^\infty" below] &&& ,
\end{tikzcd}
\end{center}
where $\Sp$ denotes the category of spectra. In particular, $\Groth_*(\taum f)$ takes values in infinite loop spaces and infinite loop maps.
\end{prop}
By applying the inverse of $\Groth_*$ to $\Groth_*(\tau_m f)$, we recover the binary operation proposed in \Cref{prop:maps_group} on the homotopy category. Moreover, \Cref{prop:general_add} shows that the group object
$$ \begin{tikzcd} \tau_m X \ar[r,"\tau_m f"] &\tau_m Y\ar[l,dashed,bend left,"\tau_m s"]\end{tikzcd}$$ as in \Cref{prop:maps_group} is a grouplike $\mathbb E_\infty$-space over $\tau_m Y$, rather than just an $H$-group object.

\subsecl{Paper outline}{subsec:outline}

\Cref{sec:group_obs} serves as a set-up for the rest of the paper, including both technical results and key examples. In \Cref{subsecl:proof_general_add}, we give proofs of \Cref{cor:add_pi_zero}, \Cref{prop:maps_group}, and \Cref{prop:general_add}. \Cref{subsecl:proof_general_add} relies on $\infone$-categorical machinery, so we summarize the key facts before our proofs. The geometrically inclined reader may choose to take the proofs of \Cref{cor:add_pi_zero} and \Cref{prop:maps_group} for granted and begin their perusal with \Cref{subsec:examples}.

The next \Cref{subsec:examples} includes the key examples of group structures on sets of vector bundles, which are the focus of the rest of the paper. This includes the groups $\mathcal G_{a_1}$ of \Cref{prop:main_construction} and $\mathcal G_{V_0}$ of \Cref{prop:main_construction2}. The proof of \Cref{cor:alg_top_agree_somewhat} uses certain elementary comparisons between different groups, which we include in \Cref{subsec:const}.

In \Cref{sec:rk2p3}, we focus on rank $2$ bundles on $\CP^3$.
 \Cref{subsec:classical_horrocks} summarizes the Atiyah--Rees $\alpha$-invariant, topological classification of rank $2$ bundles on $\CP^3$, and Horrocks' construction. In \Cref{subsec:proof_main_contruction}, we complete the proof of  \Cref{prop:main_construction} by showing that $c_2$ is a group homomorphism. \Cref{subsec:proof_rk2_thm} includes the proof of \Cref{cor:alg_top_agree_somewhat}.

\Cref{sec:generalization} focuses on group structures on isomorphism classes of $3$ bundles on $\CP^5$. In \Cref{subsec:proofrk3p5} we prove part (iii) of \Cref{prop:main_construction2}. We study group structures on rank $3$ bundles in more detail in \Cref{subsecl:interesting}. We show that, in some of the groups $\G_{V_0}$ from \Cref{prop:main_construction2}, indecomposable bundles are generated by sums of line bundles under the addition operation in the group. We view this as a proof of concept: our addition operation produces interesting bundles from simple ones. We end \Cref{subsecl:interesting} with naturally arising questions about these groups of unstable vector bundles.

We also include Appendix~\ref{app:grothendieck} which supplies details of the pointed, $\infone$-categorical Grothendieck construction for spaces. This material is used in \Cref{subsecl:proof_general_add}, although the section can be read without the appendix. This material is not new, but our exposition consolidates the necessary background and provides a reference for more comprehensive literature.

\subsecl{Acknowledgements}{subsec:acknowledgements}

First and foremost I want to thank my Ph.D. advisor, Mike Hopkins, for suggesting this project and for his wisdom and support throughout my Ph.D. program. I am also grateful to Haynes Miller for his mentorship during my time in graduate school; and to both Haynes and Elden Elmanto for serving on my dissertation committee. Mike Hill's guidance and encouragement -- mathematical and practical -- were invaluable for me while improving and revising this work. This work benefited greatly from my conversations with Aravind Asok, Lukas Brantner, Hood Chatham, Jeremy Hahn, Yang Hu, Brian Shin, Alexander Smith, and Dylan Wilson. 

The author would also like the thank the referee for extremely detailed and helpful comments on earlier drafts of this work.

While working on this project, the author was supported by the National Science Foundation under Award No.~2202914.

\subsecl{Conventions}{subsec:conventions}
\begin{itemize}
\item ``Vector bundle" refers to a complex, topological vector bundle unless otherwise stated. ``Rank" refers to complex rank.
\item We write $\s$ for the category of spaces and $\op{Sp}$ for the category of spectra. We write $\s_*$ for pointed topological spaces.
\item Limits and colimits are implicitly homotopy limits and colimits. 
\item We write $\mathcal O(k)$ for the vector bundle on $\CP^n$ with first Chern class equal to $k$. 
\item Given a classifying space $\BU(n)$, we write $\gamma_n$ for the universal bundle on it.
\item We write $\underline{\C}$ for the trivial rank $1$ vector bundle on any space.

\item Given an $\infone$-category $C$ and objects $X,Y$ in $C$, we write $\op{Maps}_C(X,Y)$ for the space of maps from $X$ to $Y$.
\item The $m$-truncation functor $\tau_{m}\: \Top \to \Top_{\leq m}$ is reflection onto the full subcategory $\Top_{\leq m}$ of spaces with $X$ with $\pi_iX=0$ for $i>m$, i.e., the left adjoint to the inclusion $\Top_{\leq m} \hookrightarrow \Top.$
\item Given spaces $X$ and $Y$, we write $[X,Y]$ for homotopy classes of maps with domain $X$ and codomain $Y$.
\item A space $Y$ is said to be $m$-skeletal if, for all $X \in \Top$, $[Y,X]\xrightarrow{\cong} [Y,\tau_mX]$ via postcomposition with the $m$-truncation map $X \to \tau_{m}X$. Examples of $m$-skeletal spaces include CW complexes with cells of dimension at most $m$ and $m$-dimensional manifolds. 
\item Fix $f\:X \to Y$ in $\s$. Given $Y \in s$ and $c\: C \to Y$, let $[C,X]_c$ denote homotopy classes of maps $g\:C\to X$ such that $f\circ g=c$.
\item We say a map $f\: X \to Y$ is $n$-connective if the homotopy fiber of $f$ has first nonzero homotopy group in degree $n$. 

\end{itemize}

\section{Background and set-up}\label{sec:group_obs}

This section has two goals: to prove necessary technical results about constructing group objects and to introduce key examples of rank-preserving group structures on sets of vector bundles. These examples will be the focus of the rest of the paper; the reader primarily interested in vector bundles may safely begin with \Cref{subsec:examples}, referencing \Cref{subsecl:proof_general_add} as needed.

In \Cref{subsecl:proof_general_add}, we prove \Cref{prop:general_add} first and then deduce \Cref{prop:maps_group} and \Cref{cor:add_pi_zero}. Much of this relies on aspects of the Grothendieck construction, or straightening and unstraightening, as described in detail by Lurie \cite[2.2 and 3.2]{HTT} and summarized in
Appendix~\ref{app:grothendieck}.

\Cref{subsec:examples} applies \Cref{cor:add_pi_zero} to maps from $\CP^3$ (respectively, $\CP^5$) into diagrams involving classifying spaces of rank $2$ (respectively, rank $3$) bundles, 
yielding group structures on collections of rank $2$ bundles on $\CP^3$ 
(respectively, rank $3$ bundles on $\CP^5$). We also show that our methods give multiple a priori distinct group structures on the same set of bundles. 

In \Cref{subsec:const}, we investigate the relationship between two different examples from the previous section. The results in this subsection are needed for the proof of \Cref{cor:alg_top_agree_somewhat}.

\subsecl{Technical preliminaries}{subsecl:proof_general_add}

In this subsection we prove \Cref{prop:maps_group}, \Cref{cor:add_pi_zero}, and \Cref{prop:general_add}, starting with the latter. 
We also record several other facts about the additive structures in question, which we will use later to study specific examples.

\begin{const}\label{note:keyproperties}
For $Y$ a space, let $(\s_{/Y})_*$ denote 
$\infone$-category of pointed objects in
the $\infone$-category of spaces over $Y$, i.e., spaces over $Y$ together with a choice of section. Let $(\s_*)^Y$ denote the 
$\infone$-category of functors from the Poincare infinity groupoid of $Y$ to $\s_*$. The following facts are justified in Appendix~\ref{app:grothendieck}.

\begin{enumerate}
\item Given a map of spaces $f\: X \to Y$ with section $s\: Y \to X$, we can naturally associate a functor $\Groth_*(f)\: Y \to \s_*,$ the {\em pointed straightening of $f$}. Heuristically, $\Groth_*(f)$ takes a point $p$ in $Y$ to the homotopy fiber $f^{-1}(p)$ pointed by $s(p)$.
\item The pointed straightening functor $\St_*$ participates in an equivalence of $\infone$-categories
\begin{center} \begin{tikzcd}
({\s_{/Y}})_* \arrow[r,bend left, "\Groth_*"] & ({\s_*})^Y.\arrow[l,bend left, "\Un_*"]
\end{tikzcd}
\end{center}
(See \Cref{cor:compatible}.)
\item There is a functor $\op{forget}\:(\s_{/Y})_* \to \s_{/Y}$ given on objects by $\left(Y \to x \right) \mapsto x$ 
and a free basepoint functor $+\: \s_{/Y} \to (\s_{/Y})_*$ given on objects by $x \mapsto ( Y \to Y \sqcup x),$ 
participating in an adjunction of $\infone$-categories:

\begin{center}
\begin{tikzcd}
\s_{/Y} \ar[r, bend right, "+" below] \ar[r, phantom, "\dashv" {labl, near end}]
& (\s_{/Y})_*\,\,. \ar[l, bend right, "\forget " above]
\end{tikzcd}
\end{center}
(See \Cref{lem:pointed_adjunction}.)
\end{enumerate}
\end{const}

\begin{proof}[Proof of \Cref{prop:general_add}] For non-negative integers $a\leq b$, let $\op{Sp}_{[a,b]}$ denote the full subcategory of spectra with homotopy groups zero outside the range $[a,b]$. Similarly, let $\s_{*,[a,b]}$ denote the full subcategory of pointed spaces with homotopy groups zero outside the range $[a,b]$ (i.e. $a$-connective, $b$-truncated spaces).

Observe that $\Groth_*(\taum f)$ takes values in $\s_{*,[n,2n-2]},$ the subcategory of pointed spaces
that are $n$-connective and $(2n-2)$-truncated. By \cite[Theorem 5.1.2]{MS},
$$\Omega^\infty\: \Sp_{[n,2n-2]} \to \s_{*,[n,2n-2]}$$ is an equivalence of $\infone$-categories. So, the factorization is automatic and there exists some functor of $\infone$-categories $A\: Y \to \operatorname{Sp}$ such that there is an equivalence \begin{equation}\label{eq:deloop_functor}\Omega^\infty A
\simeq \St_*(\taum f)\end{equation} in the $\infone$-category of such functors. This factorization is given by applying the inverse to $\Omega^\infty\: \op{Sp}_{[n,2n-2]} \to \s_{*,2n-2}$ to $\Groth_*(\taum f)$ and is therefore functorial.
\end{proof}
\begin{cor}\label{cor:unstraightened_group_ob} With set-up as in  \Cref{prop:general_add}, the diagram
$$\begin{tikzcd} \tau_m X \arrow[r," \tau_m f"] & \tau_m Y\arrow[l,dashed, bend left, " \tau_m s"] \end{tikzcd}$$ is an group object in the category of $m$-truncated spaces over $\taum Y$. The binary operation is given by applying $\op{Un}_*$ to the operation on the group object $\Groth_*(\taum f)$. \end{cor}
\begin{proof}[Proof of \Cref{cor:add_pi_zero}]
Let $c\: C\to Y$ be a map from an $m$-skeletal space $C$ to our base space $Y$, where $m=2n-2$. Let $c_+$ denote the object\begin{center}
\begin{tikzcd}
C \sqcup \taum Y \arrow[r,"c" below] &
\taum Y \arrow[l, dashed, bend right]
\end{tikzcd}
\end{center}
given by applying $+$ as in \Cref{note:keyproperties}(3) to the object $\taum c\: C \to \taum Y $.

\Cref{prop:general_add} implies that the space
$$\operatorname{Maps}_{(\s_{*})^{\taum \!Y} }\left(\Groth_*(c_+),\Groth_* \left( \taum f\right)\right)$$
possesses the structure of a grouplike $\mathbb E_\infty$ space, i.e. infinite loop space.

By \Cref{note:keyproperties}(2) we have natural equivalences of grouplike $\mathbb E_\infty$-spaces
\begin{align*}\operatorname{Maps}_{{(\s)_*}^{\taum \!Y} }\Big(\St_*(c_+),\St_* \big( \tau_{\leq m} f\big)\Big) &\simeq \operatorname{Maps}_{ (\s_{/\taum Y})_*}\big(\Un_* \St_*(c_+), \Un_* \Groth_*(\taum f)\big) \\
&\simeq \operatorname{Maps}_{ (\s_{/\taum Y})_*} \big(C \sqcup \taum Y, \taum X\big)\\&\simeq \operatorname{Maps}_{ \s_{/\taum Y}} \big(C, \taum X\big),\end{align*}
where the third step is the definition of $c_+$ and \Cref{note:keyproperties}(3).

This implies
\begin{align*}\pi_0\operatorname{Maps}_{ \s_{/\taum Y}} \big(C, \taum X\big)& \simeq \pi_0\operatorname{Maps}_{ \s_{/Y}} \big( \sk^m C , X\big) \\
&\simeq \pi_0\operatorname{Maps}_{\sy} \big( C , X\big)\\ &\simeq [C,X]_c,\end{align*}
giving the group structure on $[C,X]_c$. Note that the identity of $\pi_0\operatorname{Maps}_{ \mathcal{S}_{/\taum Y}} \big(C, \taum X\big)$ is obtained by post-composing $c\: C \to \tau_m Y$ with $s\: \tau_m Y \to \tau_m X$. Tracing through the isomorphisms above, this shows that $s \circ c\: C \to X$ is the identity in $[C,X]_c$.
\end{proof}

We record one additional result that will be useful later.
\begin{lem}\label{lem:loop_map}
Suppose that we have a commutative solid diagram
\begin{center}
\begin{tikzcd}
X\ar[dr,"\,\,\,f" right]\ar[rr,"h"] && X',\ar[dl,"f'" left]\\
& Y\ar[ul,bend left,dashed, "s"]\ar[ur, bend right,dashed, "s'" below] &
\end{tikzcd}
\end{center}
where $f$ and $ f'$ are both $n$-connective, and $s$ and $s'$ are sections of $f$ and $f'$, respectively.
Let $m= 2n-2.$
Then $\Groth_*(\taum h)\: \Groth_*(\taum f) \to \Groth_*(\taum f')$ is a natural transformation through infinite loop maps.
\end{lem}
\begin{proof} Note that the functor $\Groth_*(\taum -) \: (\s_{/Y})_* \to (\s_*)^{\tau_m Y}$ factors through $\op{Sp}^{\tau_m Y}$.
\end{proof}
This immediately implies:
\begin{cor}\label{cor:functoriality_fixed_base} With set-up as in \Cref{lem:loop_map}, let $C$ be an $m$-skeletal space and let $c\: C\to Y$ be given. The map $h\circ -\: [C,X]_c \to [C,X']_c$ is a group homomorphism.
\end{cor}

We now explicitly describe the group operation of \Cref{cor:unstraightened_group_ob} in the homotopy category of spaces by unraveling the binary operation on $\Un_*\Groth_*(\tau_mf)$.

\begin{proof}[Proof of \Cref{prop:maps_group}]
Recall the $\fold$ map from \Cref{defn:basic}(c). From the proof of \Cref{prop:general_add}, Equation~\eqref{eq:deloop_functor} supplied a functor 
$A\: \taum Y \to \op{Sp}$ such that $\Groth_*(\taum f) = \Omega^\infty A$. The group object structure on $A$ is given by
$$A \times A \simeq A \sqcup A \xrightarrow{\fold}A.$$
Applying $\Un_*$, we have a homotopy commutative diagram

\begin{equation}\label{eq:adds_agree}
\begin{tikzcd}[row sep =.2in, column sep = .35in]
\taum X \times_{\taum Y} \taum X\arrow[d,"\simeq"]\arrow[r,"\simeq"]& 
\Un_*\!\Omega^\infty (\!A\!\times\! A) &
\Un_*\! \Omega^\infty (\!A\! \sqcup \! A)\arrow[dr, phantom,"\star"]\arrow[l,"\simeq"]\arrow[r,"\Un_*\!\Omega^\infty\!\fold\!"]\arrow[d,"\simeq"] & 
\Un_*\! \Omega^\infty A\arrow[d,"\simeq"]\\
\taum (X\! \times_Y \!X) &
 &
\taum X \! \cup_{\taum \! Y } \! \taum X\arrow[r,"\fold"]\arrow[d,"\dagger"]\arrow[dr,phantom,"\star\star"] 
& \taum X\arrow[d,"\dagger"]\\
\taum(X\! \cup_Y \!\!X)\arrow[u,"\simeq", "\tau_m(sf\times 1 \cup 1 \times sf)" right]\arrow[rr,"*"] 
& 
&\taum (\taum X \! \cup_{\taum \! Y} \!\taum X )\arrow[r,"\taum\!\small\fold" ] 
& \taum X,
\end{tikzcd}
\end{equation}
where the maps labeled $\dagger$ above are the natural map from a space into its truncation and the map
labeled $*$ is the $\taum$-truncation of the map on pushouts induced by the map of spans
\begin{equation}\label{eq:star}
\begin{tikzcd}
X\arrow[d] & Y\arrow[r,"s" above]\arrow[l,"s"]\arrow[d] & X\arrow[d]\\
\taum X & \taum Y\arrow[r,"\taum s" above]\arrow[l,"\taum s"] & \taum X,
\end{tikzcd}
\end{equation}with the vertical arrows are again the natural maps from a space to its truncation.

In Diagram~\eqref{eq:adds_agree}, arrow $\tau_m(sf\times 1 \cup 1 \times sf)$ is an equivalence by the hypothesis that $f$ is $n$-connective. Inspecting the diagram, it is clear that the two right rectangles labeled $\star$ and $\star\star$ commute up to homotopy, by naturality. One then checks that starting at the bottom left corner, going up and around the large left rectangle produces the map $*$
induced by \eqref{eq:star}.

The composite obtained by starting at the upper left-hand corner of the outer rectangle, going down by two, and then right by two computes the proposed operation on $\taum X$ over $\taum Y$. On the other hand, starting in the upper left-hand corner but going first
right by two and then down by two computes the group object structure obtained in \Cref{cor:unstraightened_group_ob}.
\end{proof}

From this description of the group structure, we obtain group homomorphisms from homotopy classes of maps over one base to homotopy classes of maps over another base. First, note the following elementary lemma.

\begin{lem} Consider diagram

\begin{equation}\label{eq:functoriality_2}
\begin{tikzcd}
X\ar[d,"f"]\ar[r,"h"] & X'\ar[d,"f'" left] \\
Y\ar[r,"j"] \ar[u, bend left, "s"]& Y'\ar[u, bend right, "s'" right]
\end{tikzcd}
\end{equation}
such that $hs \simeq s'j$, $jf \simeq f'h$ and such that $s$ and $s'$ are sections of $f$ and $f'$, respectively. Such a diagram gives rise to a homotopy commutative diagram
\begin{center}
\begin{tikzcd}[row sep=.1in]
X \times_Y X \ar[d]& X \cup_Y\ar[l]\ar[r]\ar[d] X & X\ar[d] \\
X' \times_{Y'} X' & X' \cup_{Y'} X'\ar[l]\ar[r] & X'.
\end{tikzcd}
\end{center}

\end{lem}

\begin{cor}\label{cor:group_hom_2} Let $h\: X \to X'$ be as in Diagram~\eqref{eq:functoriality_2}. Suppose also that $f$ and $f'$ are $n$-connective and let $m=2n-2$. Let $C$ be $m$-skeletal and let $c\: C \to Y$ be given.
Then postcomposition with $h$ induces a group homomorphism $h\circ -\: [C,X]_c \to [C,X']_{j\circ c}.$
\end{cor}

\subsecl{Group structures on vector bundles}{subsec:examples}

The goal of this section is to introduce and explore examples of groups comprised of unstable vector bundles. First, we note the following general procedure for obtaining group structures.
\begin{example}\label{ex:general_example} Let $g\: X \to Y$ be a map of connected spaces. Consider the diagram

\begin{center}
\begin{tikzcd}
X \times_Z Y\ar[r]\ar[d, "p_1"] & Y\ar[d, "\tau_k"] \\
X\ar[r, "\tau_k \circ g"]\ar[u,dashed, bend left, "s_g" left] & \tau_k Y ,
\end{tikzcd}
\end{center}
where $Z:= \tau_k Y$, $p_1$ is the natural projection, and the section $s_g$ from $X$ to the pullback is induced by the the identity on $X$ and $g\:X \to Y$. The map $p_1$ is $k+1$-connective since $\tau_k$ is. By \Cref{cor:unstraightened_group_ob}, $\tau_{2k}f\:\tau_{2k}\left(X \times_Z Y\right) \to \tau_{2k} X$ is a group object in $2k$-truncated spaces over $\tau_{2k}X$. So, for any $2k$-skeletal space $C$ with $c\: C\to X$, the set $[C,X]_c$ inherits a group structure.
\end{example}

The next two examples are the ones that relate to Horrocks' construction \cite{Hor}.
\begin{example}\label{ex:sum_nontrivial} The map $\det(\gamma_2)\: \BU(2) \to \BU(1)$ is $4$-connective and admits a section $\left(\gamma_1\otimes \mathcal O(\smallminus b)\right) \oplus \mathcal O(b)$ where $b\in \Z$ is arbitrary. If we fix $a_1 \in \mathbb Z$, this gives a group structure on isomorphism classes of rank $2$ bundles on $\CP^3$ with $1$-st Chern class equal to $a_1$. The identity $\mathcal O(a_1\smallminus b)\oplus \mathcal O(b).$
\end{example} 
As a special case, setting $b=0$, we obtain the following:
\begin{example}\label{ex:sum_trivial} The determinant map $\det(\gamma_2)\: \BU(2)\to \BU(1)$ is $4$-connective and admits a section represented by the bundle $\gamma_1\oplus \underline \C$ on $\BU(1)$. If we fix $\mathcal O(a_1)\: \CP^3 \to \BU(1)$ for $a_1\in \mathbb Z$, this gives a group structure on isomorphism classes of rank $2$ bundles on $\CP^3$ with $1$-st Chern class equal to $a_1$. 
The identity is $\mathcal O(a_1)\oplus \underline{\C}.$
\end{example}

We will explore the previous two examples and how they are related in \Cref{subsec:const}, so we establish notation for each.

\begin{defn}\label{not:particular_sums} Let $+_{a_1,b}$ denote the operation on rank $2$ bundles on $\CP^3$ with first Chern class $a_1$ given in \Cref{ex:sum_nontrivial}. We write $\G_{a_1,b}$ for this group. In the case $b=0$ (i.e., \Cref{ex:sum_trivial}), we write $+_{a_1}$ for the operation and $\G_{a_1}$ for the group.
\end{defn}
\begin{const}\label{rmk:iso}Let $b=-\frac{a_1}{2}.$ Consider the commutative diagram
\begin{equation}\label{eq:two_adds_compare}
\begin{tikzcd}[column sep=6em]
\BU(2)\arrow[d,"\det(-)" right]\ar[r,"(-)\otimes \mathcal O(\smallminus b)"] & \BU(2) \arrow[d,"\det(-)" left]\\
\BU(1)\arrow[u,dashed,bend left, "(-) \oplus \underline \C"]\ar[r,"\operatorname{id}" below] & \BU(1) \ar[u, dashed, bend right, "\left(-\otimes \mathcal O(\smallminus b)\right)\oplus \mathcal O(b)" right]
\end{tikzcd}
\end{equation}
where for maps we write the functor represented on vector bundles. By \Cref{cor:group_hom_2}, Diagram~\eqref{eq:two_adds_compare} induces a group isomorphism 
$\phi\: \G_{a_1} \to \G_{0,b}$ given on bundles by $$V\mapsto V \otimes \mathcal O(\smallminus b).$$ 
\end{const}

Next, we introduce group structures for isomorphism classes of rank $3$ bundles on $\CP^5$, which will be our subject in \Cref{sec:generalization}.
\begin{example}\label{ex:rk3p5}
Consider the following diagram
\begin{equation}\label{eq:tilde_diagram1}
\begin{tikzcd}[column sep=3em]
\tBU(3)\arrow[dr, phantom, near start, "\scalebox{1}{\color{black}$\lrcorner$}"]\arrow[d,"v"]\arrow[r]& \BU(3)\arrow[d,"c_1\times c_2"]\\
\BU(2)\arrow[r,"c_1\times c_2"]\arrow[u,bend left,dashed, "s"] & K(\Z,2) \times K(\Z,4),
\end{tikzcd}
\end{equation}
where:
\begin{itemize}
\item $c_1\times c_2$ is the product of the first Chern class and second Chern class;
\item $\tBU(3)$ is the indicated homotopy pullback; and
\item the dotted arrow $s$ to the pullback induced by the identity map on $\BU(2)$ and by $\gamma_2 \oplus
\underline{\C}\: BU(2)\to BU(3)$.
\end{itemize}
The map $v$ 
is $6$-connective. By \Cref{cor:add_pi_zero}, for any fixed $V_0 \: \CP^5 \to \BU(2)$, the set of homotopy classes of lifts $[\CP^5,\tBU(3)]_{V_0}$ carries a group structure. 
A priori, an element in $[\CP^5, \tBU(3)]_{V_0}$ consists of a homotopy class 
$\CP^5\to \BU(3)$ and a homotopy witnessing $c_i(V)=c_i(V_0)$ for $i=1,2$. 
Such homotopies are a torsor for $H^1(\CP^5;\Z) \times H^3(\CP^5;\Z) \simeq 0$. 
Therefore the underlying set of this group is in fact isomorphism classes of rank $3$ bundles on $\CP^5$ with the same first and second Chern class as $V_0$. The identity is $V_0 \oplus \underline{\mathbb C}$.
\end{example}
\Cref{ex:rk3p5} is our focus in \Cref{sec:generalization}, 
so we introduce the following terminology.
\begin{defn}\label{GV0} Let $V_0$ be a fixed rank $2$ bundle on $\CP^3$. Let $+_{V_0}$ denote the operation on rank $3$ bundles on $\CP^5$ with first and second Chern class equal to those of $V_0$, as defined by \Cref{ex:rk3p5}. We write $\mathcal G_{V_0}$ for this group.
\end{defn}
\begin{conv}We will refer to the binary operations of \Cref{not:particular_sums} and \Cref{GV0} as {\em topological Horrocks additions}. \end{conv}

\begin{rmk} To generalize topological Horrocks addition
 from rank $2$ bundles to rank $3$ bundles, one might hope to consider isomorphism classes of rank $3$ vector bundles on $\CP^5$ with fixed $c_1$ and $c_2$. However, the additional choice of $V_0$ is necessary. To apply our setup, we would need a section $\sigma$ of $c_1\times c_2\: BU(3) \to K(\Z,2)\times K(\Z,4)$.

In 
$H^*(\BU(3);\Z/3\Z)$, $ P^1(c_2) = c_1^2c_2 + c_2^2 -c_1c_3.$ In 
$H^*(K(\Z,2) \times K(\Z,4);\Z/3\Z)$, $P^1(\iota_4)=0.$
Since $(c_1\times c_2)^*(\iota_{2i}) = c_i$, for $i=1,2$, 
existence of $\sigma$ would force
 $\sigma^*(P^1(c_2))=P^1(\iota_4)$ and $\iota_2^2\iota_4 + \iota_4^2=0,$ a contradiction.\end{rmk}

\subsecl{Properties of $+_{a_1}$ and $+_{a_1,b}$}{subsec:const}

Recall from \Cref{not:particular_sums} that $+_{a_1}$ and $+_{a_1,b}$ are binary operations, with different identities, on the set of isomorphism classes of rank $2$ vector bundles over $\CP^3$ with first Chern class $a_1$. 
\begin{lem} Let $V,W,Z$ be rank $2$ vector bundles on $\CP^3$ with first Chern class $a_1$. Then
$$(V+_{a_1}W)+_{a_1,b}Z = V+_{a_1}(W+_{a_1,b} Z).$$
\end{lem}
\begin{proof} Consider the commutative diagram:

\begin{equation}\label{eq:zig_zag_pushout}
\begin{tikzcd}[column sep=2em, row sep=3em]
& \BU(1)\arrow[dl," \gamma_1\oplus \underline \C"] \arrow[dr," \gamma_1\oplus \underline \C " ] 
& 
& \BU(1)\arrow[dr,"g'"]\arrow[dl,"g'"]\\
\BU(2)\ar[dr,"\!\!\det \gamma_2"]
&
& \BU(2)\ar[dl,"\det\gamma_2"]\ar[dr,"\det\gamma_2 "]
&
& \BU(2)\ar[dl,"\det \gamma_2 "]\\
& \BU(1) & & \BU(1),
\end{tikzcd}
\end{equation}
where $g':=\gamma_1\otimes \mathcal O(\smallminus b) \oplus \mathcal O(b)$.
 To further simplify notation:
\begin{align*} g&:=\gamma_1\oplus \underline \C & 
d & :=\det \gamma_2\\
B_2&:=\BU(2) &B_1&:=\BU(1). & &\end{align*} Diagram~\eqref{eq:zig_zag_pushout} gives a homotopy commutative diagram

\[
\begin{tikzcd}[column sep=.04in]
& &B_2 & \\ &
B_2\cup_{B_1,g'} B_2 \arrow[ur,"\op{fold}"]
& & B_2 \cup_{B_1,g} B_2\arrow[ul, "\op{fold}"] & \\ x&
\left(B_2{{\cup}}_{B_1,g} B_2\right)\cup_{B_1,g'} B_2\ar[d, "\simeq_{\leq 6}"]\ar[u,"\op{fold}\cup \op{id}"] & B_2 \cup_g B_2 \cup_{g'} B_2\ar[r,"\simeq"]\ar[dd]\ar[l,"\simeq"]
& B_2{{\cup}}_{B_1,g}\left( B_2\cup_{B_1,g'} B_2 \right)\ar[d, "\simeq_{\leq 6}"]\ar[u,"\op{id}\cup \op{fold}"] & \!y \\ &
\left(B_2{\times}_{B_1,d} B_2\right)\times_{B_1,d} B_2
& & B_2{{\times}}_{B_1,d} \left( B_2\times_{B_1,d} B_2\right)&  \\ &
&B_2\times_{B_1,d} B_2\times_{B_1,d} \times B_2 \ar[ur,"\simeq"]\ar[ul,"\simeq"] 
\ar[uuuu,dashed,rounded corners,
            to path={ -- ([yshift=-2ex]\tikztostart.south) -|
            ([xshift=37ex]\tikztotarget.east) -- (\tikztotarget)}]
\ar[uuuu,dashed, rounded corners,
            to path={ -- ([yshift=-2ex]\tikztostart.south) -|
            ([xshift=-37ex]\tikztotarget.west) -- (\tikztotarget)}]& \\ &
\end{tikzcd}
\]
where $\simeq_{\leq 6}$ indicates a map that becomes an equivalence after applying the $6$-truncation functor and the dashed arrows $x$ and $y$ are the composites defined after applying $\tau_6$ to the entire diagram. 
The fact that\[\pi_0\op{Maps}_{\s}(\CP^3,-)(x)\cong\pi_0\op{Maps}_{\s}(\CP^3,-)(y)\] proves the lemma. \end{proof}

This implies:
\begin{cor}\label{cor:additons_related_formula} Let $V,W$ be rank $2$ vector bundles on $\CP^3$ with $c_1(V)=c_1(W)=a_1$. Then
$$V+_{a_1,b} W= V+_{a_1} W -_{a_1} \left(\mathcal O(a_1-b)\oplus \mathcal O(b)\right).$$
\end{cor}
\begin{proof} This is a special case of a general fact. Suppose that $+,*$ are two abelian group structures on a set $S$ such that for all $x,y,z\in S$, $$(x+y)* z = x+(y * z).$$ Let $e_{+}$, $e_*$ be the respective identities. Then
\begin{align*}
(x+e_{*})*(y+e_{*})
&=(x+e_{*})*(e_*+y)\\
&=x+(e_{*} * e_{*})+y\\
&=x+y+e_*.\end{align*}
Set $x=x+(-e_{*})$ and $y=y+(-e_{*})$ in the above formula, where $-e_*$ is the additive inverse of $e_*$. This yields:
\begin{align*} x*y
&=x-e_*+y-e_*+e_*\\
&=x+y-e_*.
\end{align*}
We get the result by applying to the set of isomorphism classes of vector bundles with first Chern class $a_1$, letting $+=+_{a_1}$ and $*=+_{a_1,b}$, and noting that $e_*=\mathcal O(a_1-b)\oplus \mathcal O(b)$.
\end{proof}

\section{Comparing group structures on rank $2$ bundles}\label{sec:rk2p3}

In the previous section, we defined the group structures $\G_{a_1}$ and $\G_{a_1,b}$ on the set of isomorphism classes of rank 
$2$ bundles on $\CP^3$ with first Chern class equal to $a_1$ (see \Cref{ex:sum_trivial}, \Cref{ex:sum_nontrivial}, and \Cref{not:particular_sums}).
Our project in this section is to relate these group structures to Horrocks' construction and 
prove \Cref{cor:alg_top_agree_somewhat}. This will involve Atiyah and Rees' classification of rank $2$ bundles on $\CP^3$, so we begin in \Cref{subsec:classical_horrocks} with a review of relevant material from \cite{AR}, including discussion of the $\alpha$-invariant. We also recall Horrocks' construction for locally free sheaves and its necessary properties. 

In \Cref{subsec:proof_main_contruction}, we prove that $c_2\: \G_{a_1} \to H^2(\CP^3;\Z)$ is a group homomorphism, completing the proof of \Cref{prop:main_construction}. Verifying additivity of $c_2$ involves only the definition of the group structure and functoriality results about the group structures of interest, as given in see \Cref{cor:functoriality_fixed_base}. We also show that the alpha invariant defines a group homomorphism $\alpha\: \G_0 \to \Z/2$.

In \Cref{subsec:proof_rk2_thm}, we show that topological Horrocks addition and Horrocks' construction produce the same underlying topological isomorphism class, when both are defined (\Cref{cor:alg_top_agree_somewhat}). We do this by verifying that both constructions have the same effect on complete invariants of isomorphism classes of rank $2$ bundles on $\CP^3$. 

\subsecl{The classification of rank $2$ bundles on $\CP^3$ and Horrocks' construction}{subsec:classical_horrocks}

In \cite{AR}, Atiyah and Rees study complex rank $2$ topological vector bundles on $\CP^3$. They define a $\Z/2$-valued invariant $\alpha$ for bundles with even first Chern class. 

\begin{const}[The $\alpha$-invariant \cite{AR}]\label{def:alpha}Let $V$ be rank $2$ bundle on $\CP^3$ with $c_1=0$. Such bundles are classified by $\BSU(2)$. The accidental isomorphism $\BSU(2)\simeq BSp(1)$ composed with stabilization gives
$$\tilde \alpha\:\BSU(2) \to BSp \simeq \Omega^\infty \Sigma^4 KO,$$
where $KO$ denotes the real $K$-theory spectrum. Thus we have a class $\tilde \alpha \in KO^4(\BSU(2))$ and we define $$\alpha_0(V):=p_*V^*(\tilde \alpha),$$ where $p_*$ is the $KO$-theory pushforward for the spin manifold $\CP^3$.
\end{const}

Atiyah and Rees extend $\alpha$ to all bundles with $c_1(V)\equiv 0 \pmod 2$ by letting \begin{equation}\label{eq:alpha_general}\alpha(V):=\alpha_0\left(V\otimes \mathcal O\left(\frac{-c_1(V)}{2}\right)\right).\end{equation}
Atiyah and Rees give various formulae for computing the $\alpha$-invariant of special bundles, e.g.,
\begin{prop}[{\cite[Theorem 7.2]{AR}}]\label{prop:delta_alpha} For $V$ a rank 2 bundle on $\CP^3$ with $c_1(V)\equiv 0 \pmod 2$, let \[\Delta(V):= \frac{c_1^2-4c_2}{4}.\] If $V$ extends to $\CP^4$ then $\alpha(V)\equiv\frac{\Delta(\Delta-1)}{12}\pmod 2.$\end{prop}
The utility of the $\alpha$ invariant is as follows.
\begin{thm}[{\cite[Theorems 1.1, 2.8, 3.3]{AR}}]\label{thm:AR_classification}
Given $a_1,a_2\in \Z$ with $a_1a_2\equiv 0 \pmod 2$, the number of isomorphism classes of rank $2$ bundles on $\CP^3$ with $i$-th Chern class $a_i$ is:
\begin{itemize}
\item equal to $2$ if $a_1\equiv 0 \pmod 2$; and
\item equal to $1$ otherwise.
\end{itemize}
In the first case, a rank $2$ vector bundle on $\CP^3$ is determined by $c_1,c_2,$ and $\alpha$.
\end{thm}

\begin{rmk} The condition $a_1a_2\equiv 0 \pmod 2$ is necessary and sufficient for two integers to be the Chern classes of complex rank 2 topological bundles on $\CP^3$.\end{rmk}

We now review Horrocks' construction for rank $2$ algebraic vector bundles.
\begin{const}[Horrocks' construction \cite{Hor}]\label{alg_horrocks}
Let $V_1$ and $V_2$ be rank $2$ locally free sheaves on $\CP^3$. Suppose that:
\begin{itemize}
\item we have isomorphisms $\wedge^2 V_1 \simeq \wedge^2 V_2\simeq \mathcal O(a)$ for some $a\in \mathbb Z$;
\item we have regular\footnote{A section $s$ of $V$ is regular if its vanishing locus is of codimension equal to the rank of $V$.} sections $s_i\: \mathcal O \to V_i^*$; and
\item the sheaves $\mathcal R_i := \op{coker}(s_i^*\:\mc F_i \to \mc O)$ have disjoint supports.
\end{itemize}
For each $i=1,2$, the Koszul complex relative to $s_i$ has the form $ 0 \to \mathcal  O(a) \to V_i\to \mathcal{O}$ and is exact since $s_i$ is regular. By definition, we have exact sequences $$0\to \mathcal O(a) \to V_i \to \mathcal O \to \mathcal R_i\to 0.$$
The sheaf $V_1+_H V_2$ is defined by the diagram
\begin{equation}\label{horrocks_add_defn}
\begin{tikzcd}
0 \arrow[r]
& \mathcal \mathcal O(a) \oplus \mathcal O(a) \arrow[r]
& V_1 \oplus V_2 \arrow[r]\arrow[from=d]
& \mathcal O \oplus \mathcal O \arrow[r]\arrow[from=d,"\Delta"] \arrow[from=dl, phantom, "\scalebox{1}{\color{black}$\llcorner$}" {rotate=180, near start}, color=black]
&\mc {R}_1 \oplus\mc{ R}_2 \arrow[r] \arrow[from=d,"\simeq"']
& 0\\
0 \arrow[r]
& \mathcal O(a) \oplus \mathcal O(a)
\arrow[r]\arrow[u,"\simeq"]\arrow[d,"\nabla"]
&\mc W \arrow[r]\arrow[d]
& \mathcal O\arrow[r]
& \mc {R}_1 \oplus\mc{ R}_2 \arrow[r]
& 0 \\
0 \arrow[r]
& \mathcal \mathcal O(a) \arrow[r]
& V_1 +_{H} V_2 \arrow[r] \arrow[from=ul, phantom, "\scalebox{1}{\color{black}$\lrcorner$}" {rotate=180, near end}, color=black]
&\mathcal O\arrow[from=u,"\simeq"]\arrow[r]
&\mc {R}_1 \oplus\mc{ R}_2 \arrow[r]\arrow[from=u, "\simeq"]
& 0 \\
\end{tikzcd}
\end{equation}
where $\mathcal W$ is the indicated pullback along the diagonal $\Delta$ and $V_1+_HV_2$ is the pushout along the fold map $\nabla$.
The middle and bottom rows in Diagram~\eqref{horrocks_add_defn} 
are both exact and $V_1+_H V_2$ is locally free of rank $2$ \cite[Theorem 1]{Hor}. The bottom exact sequence of Diagram~\eqref{horrocks_add_defn} implies that
\[c_1(V_1 \horro V_2) = c_1 (V_1) = c_1(V_2) =a.\] \end{const}

Atiyah--Rees and Horrocks study the effect of $+_H$ on algebraic invariants:

\begin{thm}[{\cite[Theorem 1]{Hor}, \cite[Corollary 5.7]{AR}}]\label{cor:alpha_almost_add_alg}
Let $V_1,V_2$ be rank $2$ algebraic bundles on $\CP^3$ with $c_1(V_1)=c_1(V_2)=-m$, with $m\geq 0$.\footnote{This is necessary for $V_i$ to admit regular sections.} Then $$c_2(V_1+_H V_2)=c_2(V_1)+c_2(V_2).$$

Furthermore, suppose that $m=2n$ with $n\geq 0$.
\begin{itemize}

\item If $n$ is odd or $n\equiv 0\pmod 4$, then $\alpha(V_2 \horro V_2) = \alpha(V_1) + \alpha(V_2) \in \Z/2\Z.$
\item If $n\equiv 2\pmod 4$, then $\alpha(V_1 \horro V_2) = \alpha(V_1) + \alpha(V_2)+1.$
\end{itemize}
\end{thm}

\subsecl{Proof of \Cref{prop:main_construction}}{subsec:proof_main_contruction}

In the previous section, we defined the group structures $\G_{a_1}$ on the set of isomorphism classes of rank $2$ bundles on $\CP^3$ with first Chern class equal to $a_1$ (see \Cref{ex:sum_trivial} and \Cref{not:particular_sums}). This gives the group structure required by part (i) of \Cref{prop:main_construction}.

For part (ii), note that for any $f\: X\to Y$ with section $s\: X \to Y$ satisfying the hypotheses of \Cref{prop:general_add}, and any $c\: C\to Y$ with $C$ an $m$-skeletal space, the identity element is given by $s\circ c$. In this case, $s \circ c \simeq \mathcal O(a_1) \oplus \underline{\C}$.

Part (iii) is a consequence of the following result.
\begin{prop}\label{prop:horrocks_add_alpha}
Let $V,W\: \CP^3 \to \BU(2)$ be two bundles with $\det V \simeq \det W \simeq \mathcal O(a_1)$. Let $+_{a_1}$ denote topological Horrocks addition as in \Cref{not:particular_sums}. Then:
\begin{itemize}
\item for any $a\in \mathbb Z$, the second Chern class $c_2$ is a homomorphism for $+_{a_1}$; and
\item if $a_1= 0$, then $\alpha$ is a homomorphism for $+_0$.
\end{itemize}
\end{prop}
\begin{proof}
Let $K(\Z,4)$ denote the Eilenberg--Mac Lane space determined by $\pi_4\left (K(\Z,4)\right)=\Z$ and $\pi_i\left( K(\Z,4)\right)=0$ for $i\neq 4$. Note that the diagram
\begin{center}
\begin{tikzcd}
\BU(2) \arrow[r, "c_2"]\arrow[d,"\det"]& K(\mathbb Z,4)\ar[d] \\
\BU(1) \arrow[r,"0"] \arrow[u,bend left, dashed, "s"] & * \arrow[u,bend right, dashed] & .
\end{tikzcd}
\end{center}
The solid diagram is homotopy commutative
 and the subdiagram with only dashed vertical arrows is homotopy commutative. 
By \Cref{cor:group_hom_2}, we get a group homomorphism
$c_2\:[\CP^3,\BU(2)]_{c} \to H^4(\CP^3,\Z).$

For the second item, consider the diagram
\begin{equation}\label{eq:two_adds_compare}
\begin{tikzcd}
\BU(2)\arrow[d,"c_1"] & \BSU(2) \arrow[l,"\iota"]\arrow[r,"\alpha"]\arrow[d] & \Omega^\infty \Sigma^4 KO\arrow[d]\\
K(\Z,2)\arrow[u,dashed, bend left]& *\arrow[l]\arrow[u,dashed,bend left]\arrow[r] & *\arrow[u,dashed,bend left]&,
\end{tikzcd}
\end{equation}
where the solid diagram is homotopy commutative and the subdiagram with only dashed vertical arrows is homotopy commutative.
All vertical homotopy fibers are $4$-connective, so we can apply \Cref{cor:add_pi_zero}
to all vertical diagrams
to obtain group structures on homotopy classes of maps from $\CP^3$ into each diagram.
By \Cref{cor:group_hom_2}, we get group homomorphisms
\begin{equation}
\begin{tikzcd}
{{[\CP^3,\BU(2)]}_{c_1=0} }
& {[\CP^3,\BSU(2)]} \arrow[l," \iota\circ (-)"]\arrow[r,"\alpha\circ (-)"]
&{[\CP^3,\Omega^\infty\Sigma^4KO]} \simeq KO^4(\CP^3).
\end{tikzcd}
\end{equation}
The map $\iota\circ -$ is an isomorphism. For $V\: \CP^3\to \BU(2)$ with $c_1(V)=0$,
$$p_*\Big(\alpha \circ \big((\iota\circ(-))^{-1}(V)\big)\Big)=\alpha(V).$$ Since push-forward on cohomology is a group homomorphism, we conclude that $\alpha$ is additive for $+_0$ as was to be shown.
\end{proof}

\subsecl{Proof of \Cref{cor:alg_top_agree_somewhat}}{subsec:proof_rk2_thm}

To show that $V+_{a_1} W \simeq V+_H W$ when both $+_{a_1}$ and $+_H$ are defined, we check the $V+_{a_1}W$ and $V+_HW$ have the same Chern classes and $\alpha$ invariant. Both $+_{a_1}$ and $+_H$ fix the first Chern class, so it suffices to show that:
\begin{itemize}
\item[(i)] $c_2(V+_{a_1}W)=c_2(V+_HW)$, and
\item[(ii)] $\alpha(V+_{a_1} W)= \alpha(V+_H W)$.
\end{itemize}
Item (i) follows from additivity of $c_2$ for both operations, by \Cref{cor:alpha_almost_add_alg} and \Cref{prop:horrocks_add_alpha}. 

Checking (ii) is complicated by the fact that the $\alpha$ invariant does not play well with $+_{a_1}$ when $a_1 \neq 0$. We bootstrap from this case $a_1=0$ to obtain a formula for $\alpha(V+_{a_1} W)$ in terms of $\alpha(V)$ and $\alpha(W)$. This involves the study of the groups $\G_{0,b}$ for $b$ nonzero (see \Cref{ex:sum_nontrivial}).

Let $a_1 \in \Z$ and suppose that $V,W$ are rank $2$ bundles on $\CP^3$ with $c_1(V)=c_1(W)=a_1$. By \Cref{prop:main_construction}(iii), $c_2(V+_{a_1} W)=c_2(V)+c_2(W)=c_2(V+_H W)$. In the case $a_1\equiv 1 \pmod 2$, \Cref{thm:AR_classification} implies $V+_{a_1}W\simeq V+_HW$ and we are done.

Now suppose $a_1\equiv 0\pmod 2$. Recall \Cref{rmk:iso} and the associated group isomorphism
$\phi\: \G_{a_1} \to \G_{0,b}$. Using the definition of $\alpha$ (see Equation~\eqref{eq:alpha_general}), \Cref{prop:horrocks_add_alpha}, and \Cref{cor:additons_related_formula}:
\begin{align*}\alpha\left(V+_{a_1} W\right)&=\alpha\left(\phi(V+_{a_1}W)\right) \\
&=\alpha\left(\phi(V)+_{0,b}\phi(W)\right) \\
&=\alpha\left(\phi(V) +_0 \phi(W)-_0 \mathcal O(\smallminus b)\oplus \mathcal O(b)\right) \\
&=\alpha(V)+\alpha(W)-\alpha\left(\mathcal O(\smallminus b)\oplus \mathcal O(b)\right).
\end{align*}
To compute $\alpha(\mathcal O(b)\oplus \mathcal O(\smallminus b))$, we use \Cref{prop:delta_alpha}.
 Since $\mathcal O(\smallminus b)\oplus \mathcal O(b)$ extends to $\CP^4$ and $\Delta\left(\mathcal O(b)\oplus \mathcal O(\smallminus b)\right)=b^2,$ we see that

\[\alpha\left(\mathcal O(\smallminus b)\oplus \mathcal O(b)\right) \equiv\frac{b^2(b^2-1)}{12}
\equiv \frac{b^2(b+1)(b-1)}{4} \pmod 2.\]
So we must determine when $b^2(b+1)(b-1)$ is divisible by $8$.
\begin{itemize}
\item If $b\equiv 0 \pmod 4$, then $16$ divides $b^2(b+1)(b-1)$ so $\alpha\left(\mathcal O(b)\oplus \mathcal O(\smallminus b)\right)=0.$
\item If $b$ is odd, then both $b+1$ and $b-1$ are even, and one of $b+1$ and $b-1$ is divisible by $4$, so $\alpha\left(\mathcal O(b)\oplus \mathcal O(\smallminus b)\right)=0$.
\item If $b\equiv 2\pmod 4$, then $b+1$ and $b-1$ are odd, and $b^2$ is divisible by $4$ but not $8$, so $\alpha\left(\mathcal O(b)\oplus \mathcal O(\smallminus b)\right)=1.$
\end{itemize}

Thus,
\[ \alpha\left(V+_aW\right)=\alpha(V)+\alpha(W)+\epsilon(a),\,\text{ where } \epsilon(a):=\begin{cases} 0 \text{ if $a \not \equiv 4\pmod 8$} \\ 1 \text{ if $a\equiv 4\pmod 8$.}\end{cases}\]

Comparing this with \Cref{cor:alpha_almost_add_alg} yields the conclusion
$\alpha\left(V+_{a_1}W\right)=\alpha\left(V+_HW\right)$.\qed
\vspace{1em}
While this completes the argument that algebraic and topological Horrocks addition produce the same underlying topological isomorphism class, our method is indirect. The reader might hope for a more geometric relationship between the constructions.
\begin{prob} Is there an explicit, bundle-theoretic comparison between topological Horrocks addition $+_{a_1}$ and Horrocks' original construction for locally free sheaves?
\end{prob}

\section{Group structures on rank $3$ bundles on $\CP^5$}\label{sec:generalization}

We now explore additive structures on rank $3$ bundles on $\CP^5$. Recall that, given a fixed rank $2$ bundle $V_0$ on $\CP^5$, \Cref{ex:rk3p5} gives a group $\G_{V_0}$ of isomorphism classes of rank $3$ vector bundles with the same first and second Chern classes as $V_0$.
The identity is $V_0 \oplus \underline \C$. Our primary interest here is to better understand the groups $\G_{V_0}$ and their properties.

We complete the proof of \Cref{prop:main_construction2} in \Cref{subsec:proofrk3p5} by showing that $c_3$ is a group homomorphism.
In \Cref{subsecl:interesting}, we explore the structure of $\G_{V_0}$ for $V_0 = \mathcal O(a) \oplus \mathcal O(b)$. 
We show that $+_{V_0}$ allows for the construction of interesting bundles from simple ones, as did Horrocks' construction on rank $2$ bundles on $\CP^3$ (see \Cref{thm:horrocks_generates}). 
\begin{prop}\label{prop:not_sum_lb} For infinitely many isomorphism classes of rank two bundles $V_0 := \mathcal O(a)\oplus \mathcal O(b)$ on $\mathbb CP^5$, there exists a bundle $W=\mathcal O(x)\oplus \mathcal O(y) \oplus \mathcal O(z)$ with $W\in \G_{V_0}$, such that the subgroup generated by $W$ contains bundles that are not sums of line bundles. Moreover, the subgroup generated by $W$ is of finite index in $\G_{V_0}$.
\end{prop}
The proof of this result is elementary, involving only the study of possible Chern classes of sums of line bundles and additivity of $c_3$.

\subsecl{The third Chern class and the structure of $\G_{V_0}$ for rank $3$ bundles}{subsec:proofrk3p5}

We first prove the additivity of $+_{V_0}$ for $c_3$ and complete a theorem stated in the introduction.
\begin{proof}[Proof of \Cref{prop:main_construction2}]
We have already established the group structure in \Cref{ex:rk3p5}. The identity element is $V_0 \oplus \underline{\C}$ by \Cref{cor:add_pi_zero}.

To show that the third Chern class $c_3$ is a homomorphism, consider the diagram
\begin{center}
\begin{tikzcd}[row sep=1em]
\tBU(3) \arrow[r, "c_3"]\arrow[d]& K(\Z,6)\ar[d] \\
BU(2) \arrow[r,"0"] \arrow[u,bend left, dashed, "s"] & * ,\arrow[u,bend right, dashed]
\end{tikzcd}
\end{center}
where the solid diagram is homotopy commutative
 and the subdiagram with only dashed vertical arrows is homotopy commutative. 
By \Cref{cor:group_hom_2}, we get a group homomorphism
$c_3\:[\CP^3,BU(3)]_{c} \to H^6(\CP^5,\Z).$
\end{proof}

In \cite{MO}, we prove that rank $3$ bundles on $\CP^5$ are determined by Chern classes except if $c_1(V)\equiv c_2(V) \equiv 0 \pmod 3$, in which case there are three pairwise non-isomorphic bundles with the same Chern classes as $V$. This proves the following result.
\begin{cor}\label{cor:kernel} The kernel of $c_3\: \G_{V_0} \to \Z$ is:
\begin{enumerate}
\item trivial if $c_1(V_0)\not\equiv 0 \pmod 3$ or $c_2(V_0)\not\equiv 0 \pmod 3$; and
\item isomorphic to $\Z/3$ if $c_1(V_0)\equiv c_2(V_0) \equiv 0 \pmod 3$.
\end{enumerate}
\end{cor}
\begin{rmk}\label{rmk:c3_splits} Note that the image of $c_3$ is an infinite subgroup of $\Z$, since the conditions for three integers to be the Chern classes of a rank $3$ bundle on $\CP^5$ are a finite number of congruences (see \cite[Lemma 2.16]{MO}). In the first case in \Cref{cor:kernel}, $\G_{V_0}$ is abstractly isomorphic to $\Z$; in the second, $\G_{V_0} \simeq \Z \oplus \Z/3$.\end{rmk}

In \cite{MO}, we define a $\Z/3$-valued invariant $\rho$ for rank $3$ bundles $V$ on $\CP^5$ such that $c_1(V) \equiv c_2(V)\equiv 0 \pmod 3$. We prove that $c_1,c_2,c_3$ and $\rho$ are complete invariants of isomorphism classes of complex rank $3$ topological bundles on $\CP^5$. Additivity of $c_3$ leads to the following:
\begin{prob} For general $V_0$ with first and second Chern class divisible by $3$, what is a formula for $\rho(V+_{V_0}W)$ in terms of $\rho(V)$, $\rho(W)$, and $\rho(V_0)$? \end{prob}
Note that a necessary condition for $\rho$ to be a group homomorphism is that $\rho(V_0\oplus \underline{\C})=0$.

\subsecl{Constructing rank $3$ bundles on $\CP^5$ from sums of line bundles}{subsecl:interesting}

We begin this subsection by showing that there exist infinitely many isomorphism classes $V_0=\mathcal O(a)\oplus \mathcal O(b)$ on $\CP^3$ such that $\G_{V_0}$ contains another sum of line bundles that is not the group identity. This is a preliminary to prove \Cref{prop:not_sum_lb}.
First, we simplify notation.
\begin{defn} Given $n_1,\dots, n_k\in \Z$, let $\mathcal O(n_1,\dots,n_k):=\mathcal O({n_1})\oplus \dots \oplus \mathcal O({n_k}).$
\end{defn}
An element $\mathcal O(x,y,z)\in\G_{\mathcal O\!(a,b)}$ 
is equivalent to an integer solutions to the equations
\begin{equation}\label{eqeq1}x+y z=a+b\end{equation}
\[ xy+yz+zx = ab.\]
Note that, since the identity of $\G_{\mathcal O\!(a,b)}$ is 
$\mathcal O(a,b) \oplus \underline{\C}$, a sum of line 
bundles $\mathcal O(x,y,z)\in \G_{\mathcal O\! (a,b)}$ is a nonidentity element if and only if $c_3\left(\mathcal O(x,y,z)\right)=xyz$ is nonzero.

Let $c:=a-x$ and $d:=b-y$. Substituting in Equations~\eqref{eqeq1}, 
we obtain the equation
 $$Q:=c^2+d^2-bd-ac+cd =0.$$
This is the equation of a quadric hypersurface in $\mathbb{P}^3_{a,b,c,d}$, whose rational points can be found via standard methods. 
For any $u,l,v,w \in \Z$ we get solutions to Equations~\eqref{eqeq1}
\begin{align*}
a&=w(u^2+v^2+uv-lv)&x &= w(v^2+uv-lv)\\
b& = wul&y &= wu(l-v)\\
&&z &= wu(u+v).\end{align*}
\begin{rmk}\label{rmk:inf_many}The previous computation shows that there are infinitely many $\G_{V_0}$ where $V_0=\mathcal O(a,b)$, such that $W=\mathcal O(x,y,z)$ is a non-identity element of $\G_{V_0}$.
\end{rmk}
\begin{example}\label{example:explicit_small_index}
Consider $W=\mathcal O(2 , \smallminus 1, 2)$ and $V_0= \mathcal O(3,0)$.
By construction, $W \in \G_{V_0}$. Since $c_1$ is even and $c_1, c_2$ are both divisible by $3$, the Schwarzenberger conditions as in \cite[Lemma 2.16]{MO}
imply that there exists a bundle $W' \in \mathcal G_{V_0}$ with $c_3(W')=a$ if and only if $a$ is even.
By \Cref{rmk:c3_splits}, $\G_{V_0} \simeq \Z\oplus \Z/3$ and, under this identification, $c_3$ is projection onto the first factor. Since $c_3(W)=4$, this implies that the subgroup generated by $W$ is index $6$.
\end{example}

\begin{proof}[Proof of \Cref{prop:not_sum_lb}]
Let $V_0={\mathcal O(a,b)}$ and $W=\mathcal O(x,y,z)$ be as in \Cref{rmk:inf_many} above with $c_3(W)\neq 0$. Inductively, let \[+_{V_0}^n W:= W+_{V_0}\left(+_{V_0}^{n-1}W\right).\]
Assume for a contradiction that $+_{V_0}^nW$ is a sum of line bundles for all positive integers $n$. 
Let $a_1,a_2,a_3 \in \Z$ denote the Chern classes of $W$. Since $+_{V_0}$ 
preserves $c_1$ and $c_2$, and $c_3\: \G_{V_0} \to H^6(\CP^5;\Z)$ is a group homomorphism,
\begin{align*}
c_1(+_{V_0}^n W) & = a_1, & c_2(+_{V_0}^nW) &=a_2, \text{ and}&
c_3(+_{V_0}^nW) &= na_3.
\end{align*}
Thus, the equations
\begin{align*}
X+Y+Z & = a_1 &XYZ &= n a_3\\
XY+YX+ZX &= a_2
\end{align*}
have a solution $(X,Y,Z)=(x_n,y_n,z_n)$ for all $n\in \Z$.
Let $p$ be prime such that $p>|3 a_i|$ for $i=1,3$. 
Take $n=p$ in the above equations.
Since $c_3(W)=a_3\neq 0$, the third equation implies that $p$ divides one of $x_p$, $y_p$, or $z_p$.
Without loss of generality, $p$ divides $x_p$. But note that
\[|z_p y_p| \leq |a_3|< \frac{p}{3}\]
and so $|z_p| < \frac{p}{3},$ and $|y_p|  < \frac{p}{3}.$
Since $|x_p|$ is divisible by $p$, it is larger than $|y_p|$, $|z_p|$, 
and $|y_p+z_p|$. Therefore
\[ |x_p + y_p + z_p|  \geq |x_p| - |y_p|-|z_p|v \geq \frac{p}{3}.\]
On the other hand,
$|x_p+y_p+z_p| = |a_1|< \frac{p}{3},$ a contradiction.

Now consider the subgroup generated $W$. Since $c_3(W)\neq 0$, $\ker(c_3) \cap \langle W \rangle $ is trivial and we have an exact sequence of groups
\[\ker(c_3) \to \G_{V_0}/\langle W \rangle  \to \Z/c_3(W).\]
Finally, $\ker(c_3)$ is finite by \Cref{rmk:c3_splits} and $\Z/c_3(W)$ is finite, so $\G_{V_0}/H$ is also finite.\end{proof}

This section gives insight into the groups $\G_{V_0}$, but many questions remain.
\begin{prob}
\begin{enumerate}
\item For groups $\G_{V_0} \simeq \Z \oplus \Z/3$, what are generators for $3$-torsion?
\item Are the Horrocks bundles of rank $3$ on $\CP^5$ elements of some $\G_{V_0}$? If so, what subgroups do they generate?
\item What can be said about the structure of $\G_{V_0}$ for $V_0$ not a sum of line bundles?
\end{enumerate}
\end{prob}

\begin{appendices}
\section{The pointed Grothendieck construction for spaces}\label{app:grothendieck}

This appendix describes the pointed version straightening and unstraightening for functors to spaces, combining facts from \cite{HTT}. The particular goal is to make precise the statement that a diagram of spaces
\begin{equation}\label{eq:groupobdiag}\begin{tikzcd} X \arrow[r," f"] & Y\arrow[l,dashed, bend left, " s"] \end{tikzcd}\end{equation}
with $f$ is $n$-connective can equivalently be viewed as a functor from $Y$ to infinite loop spaces, after $(2n-2)$-truncation.

\begin{notation}\label{not:groth}
We establish terminology and conventions and provide references.
\begin{enumerate}
\item Let $\sset$ denote the category of simplicial sets and $\operatorname{Cat}_{\Delta}$ denote the category of simplicially enriched categories.
\item By $\infone$-category, we will mean a simplicial set that satisfies the inner horn condition (a quasi-category, \cite[1.1.2.4]{HTT}).
\item Let $\mathfrak C$ denote the left adjoint to the homotopy coherent nerve functor $N\: \op{Cat}_{\Delta} \to \operatorname{sSet}$ \cite[1.1.5]{HTT}. There is a Quillen equivalence of simplicial model categories
\begin{center}
\begin{tikzcd}
\operatorname{Cat}_{\Delta} \ar[r, bend right, "N" below]
& \sset\ar[l, bend right, "\mathfrak C " above]
\end{tikzcd}
\end{center}
between simplicial sets with the Joyal model structure and simplicially enriched categories with the Bergner model structure (see \cite[Section 1.1.5]{HTT} for discussion; \cite{Berg2}, \cite{Berg1}, and \cite{Joyal1} for a proof). 
\item Let ${\sset}_{/S}$ denote the simplicial model category of simplicial sets over $S$, endowed with the contravariant model structure. We have an associated simplicial category $\operatorname{ RFib}(S)$ obtained by taking fibrant-cofibrant objects.
\item Let $\Kan$ denote the $\infone$-category of Kan complexes obtained by taking fibrant-cofibrant objects in the simplicial model category of simplicial sets with the Kan model structure and applying $N$. Given a Kan complex $S$, let $\Kan_{/S}$ denote $\infone$ over-category of Kan complexes over $S$. 
\item Given a  $\infone$-category $C$ and a simplicial set $S$, we write $\Fun(S,C)$ for the simplicial set of maps from $S$ to $C$.
This is an $\infone$-category since $C$ is, and models functors from $S$ to $C$ \cite[1.2.7.2]{HTT}.
\item We ignore set-theoretic issues and refer the concerned reader to \cite[1.2.15]{HTT}. 
\end{enumerate}
\end{notation}
The fundamental result that we need is the following.
\begin{thm}[Lurie, \cite{HTT}]\label{thm:lurie_groth} Let $S$ be a Kan complex. There is an equivalence of $\infone$-categories
\begin{equation}\label{diag:Kan} \begin{tikzcd}
\Kan_{/S} \arrow[r, bend right, "\St"] & \Fun(S,\Kan) \arrow[l, bend right, "\Un"].
\end{tikzcd}
\end{equation}
\end{thm}
\begin{proof} By
\cite[2.2.3.11]{HTT}, for any simplicial set
we have an equivalence of simplicial categories 
\begin{equation}\label{eq:frak} \begin{tikzcd}
\operatorname{RFib}(S) \arrow[r, bend right, "\St"] 
& ({\op{sSet}}^{ {\mathfrak{C}[S]}^{\op{op}} })^\circ, \arrow[l, bend right, "\Un"]
\end{tikzcd}
\end{equation}
where $\op{sSet}$ has the Kan model structure, ${\op{sSet}}^{ {\mathfrak{C}[S]}^{\op{op}} }$ has
 the projective model structure, and $(-)^\circ$ indicates we take fibrant-cofibrant objects.
Since $S$ is a Kan complex, \cite[3.1.5.1(A3)]{HTT} implies $\Kan_{/S}\simeq N(\RFib(S))$ as $\infone$-categories. 

By \cite[4.2.4.4]{HTT}, applying the homotopy coherent nerve to the right side of Equation~\eqref{eq:frak}
 recovers $\op{Fun}(S^{\op{op}}, \op{Kan}).$ By \cite[Section 57]{Rezkqcat} combined with \cite[Proposition 14.14]{Rezkqcat}, every Kan complex is equivalent to its opposite as an $\infone$-category.
 \end{proof}
To obtain a version for spaces over a base with a section, we take pointed objects on either side of \Cref{thm:lurie_groth}. 

\begin{defn} Given an $(\infy,1)$-category $C$ with final object $t$ \cite[1.2.12]{HTT}, 
$C_*$ denote the full $\infone$-subcategory of $\Fun(\Delta^1,C)$ spanned by functors that restrict to $t$ on $0$.
Explicitly, if $\op{fib}_t$ denotes the (1-categorical) fiber in $\sset$ at $t$, then
$$C_* := \operatorname{fib}_{t}\big(\Fun(\Delta^1,C) \to \Fun(\Delta^0, C)\big)$$ 
where the map is given by pre-composition with $\{0\} \simeq \Delta^0 \hookrightarrow \Delta^1.$ \end{defn}

\begin{rmk}Given an $\infone$-category $C$ with final object $t$, 
the infinity category $C_*$ is also modeled by the $\infone$-under category $C_{t/}$. 
See \cite[1.2.9]{HTT} for a discussion of $\infone$-under categories and \cite[4.2.1.5 and 7.2.2.8]{HTT} for the equivalence.\end{rmk}
Since equivalent $\infone$-categories support the same category theory, the equivalences $\St$
and $\Un$ will induce one on pointed objects.  

\begin{lem}\label{lem:equiv_pointed} An equivalence of $\infone$-categories induces an equivalence upon taking pointed objects. \end{lem}

\begin{proof} This is the dual statement of \cite[1.2.9.3]{HTT} (dualized, for example, following the method of \cite[1.2.9.5]{HTT}), specialized to the case of a final object.
\end{proof}
We can apply \Cref{lem:equiv_pointed} directly to the equivalences of \Cref{thm:lurie_groth}, but to understand the result we first describe pointed objects on the left side of Diagram~\eqref{diag:Kan}.

\begin{lem} For any $\infone$-category $C$ with a final object $t$, the $\infone$-category $\Fun(S,C)$ has a final object the constant functor at $t$. There is an equivalence of $\infone$-categories $\Fun(S,C)_*\simeq \Fun(S,C_*)$.
\end{lem}
\begin{proof} The first statement follows from the dual of \cite[5.1.2.3]{HTT}.
For the second, recall that $\Fun(S,C)_*$ is modeled by the fiber at the final object of the map
$$ \Fun\left(\Delta^1,\Fun(S,C)\right) \to \Fun\left(\Delta^0, \Fun(S,C)\right)$$ given by restricting to $\{0\} \hookrightarrow \Delta^1$.
Moreover, there is an adjunction $ S \times -\dashv \Fun(S,-)$ as enriched functors $\sset \to \sset$, since $\Fun(S,-)$ is an internal hom to in the monoidal category $(\sset, \times, \Delta^0)$. So we get isomorphisms of simplicial sets:
\begin{align*}&\operatorname{fib}\left( \Fun\left(\Delta^1,\Fun(S,C)\right) \to \Fun\left(\Delta^0, \Fun(S,C)\right)\right) \\ &\simeq \operatorname{fib}\left(\Fun\left(\Delta^1\times S,C\right) \to \Fun
\left(\Delta^0\times S,C
\right)\right) \\
&\simeq \operatorname{fib}\left(\Fun\left(S, \Fun(\Delta^1, C)\right) \to \Fun\left( S,\Fun(\Delta^0,C)\right)
\right).\end{align*}
Since right adjoints preserve limits we have that:
\begin{align*}\operatorname{fib}\big(\Fun\left(S, \Fun(\Delta^1, C)\right) \!\to\! \Fun\left( S,\Fun(\Delta^0,C)\right)\big)  \simeq \Fun\left(S, \op{fib} \left(\Fun(\Delta^1,C)\! \to\! \Fun(\Delta^0, C)\right) \right).\end{align*}
Tracing through our equivalences, we have shown that $\Fun(S,C)_* \simeq \Fun(S, C_*)$.
\end{proof}

\begin{cor}\label{cor:compatible}

Straightening and unstraightening induce mutually inverse an equivalences of $\infone$-categories $\St_*\: {(\Kan_{/S}})_* \to \Fun(S,\Kan_*)$ and $\Un_*\:\Fun(S,\Kan_*) \to {(\Kan_{/S}})_* .$
\end{cor}

We record one more result about pointed objects in $\infone$-categories.
\begin{lem}\label{lem:pointed_adjunction} Let $\mathcal C$ be a $\infone$-category with a final object $t$ and with coproducts. Then there is an adjunction of $\infone$-categories
\[\begin{tikzcd}
C \ar[r, bend right, "+" below] \ar[r, phantom, "\dashv" {labl, near end}]
& C_* \ar[l, bend right, "\forget " above]\end{tikzcd}
\]
\end{lem}
\begin{proof}
We write $\op{F}$ for the forgetful functor. This functor is given on objects by $(t \to Y\big) \mapsto Y$.
The functor $+$ is given on objects by $c \mapsto c \sqcup t,$ where $\sqcup$ is the $\infone$-coproduct in $C$ \cite[4.4.1]{HTT}. 

To show $+$ is left-adjoint to $\op{F}$, we use \cite[5.2.2.8]{HTT}: it suffices to provide a unit transformation $$u\:\operatorname{Id}_C \to \op{F} \circ +$$ in $\Fun(C,C)$ such that a certain associated composite is an isomorphism in the homotopy category. Let $u_c\: c \to t \sqcup c$ be the structure map of the coproduct. We must show

\begin{center}
\begin{tikzcd} \Maps_{C_*}(t \to c \sqcup t, t \to y) \arrow[r," \op{apply \,\, F}"] & \Maps_{C}\big(\op{F}(t \to c \sqcup t), \op{F}(t \to y)\big) \arrow[r, "u_c \circ -"] &\Maps_C(c,y)\end{tikzcd}\end{center}
is an isomorphism in the homotopy category of spaces for all $c$ and $y$ in $C$.
Combining \cite[5.5.5.12]{HTT} with \cite[Lemma 7.2.2.8; Proposition 4.2.1.5]{HTT}, we have an equivalence of spaces:
$$\Maps_{C_*}(t \to c \sqcup t, t \to y) \simeq \operatorname{hofib}\Big(\Maps_C(c \sqcup t, y) \to \Maps_C(t,y) \Big),$$
where the arrow $\Maps_C(c \sqcup t, y) \to \Maps_C(t,y) $ is induced by precomposing with the given map
$t \to c \sqcup t,$ and the fiber is taken over
$t \to y\in C_*$ viewed as an object in $\Maps_C(t,y)$. Consider the homotopy commutative diagrams of spaces:
\begin{center}
\begin{tikzcd}[row sep=.2in]
\Maps_C(c \sqcup t, y) \arrow[r]\arrow[d,"\simeq"]& \Maps_C(t,y)\arrow[d,"="] \\
\Maps_C(c,y) \times \Maps_C(t,y) \arrow[r, "\pi_2"] &\Maps_C(t,y),
\end{tikzcd}
\end{center}
This implies
\begin{align*}\Maps_{C_*}\big(t \to c \sqcup t, t \to y\big) & \simeq \operatorname{hofib}\big(\pi_2\: \Maps_C(c,y) \times \Maps_C(t,y) \to \Maps_C(t,y)\big)\\
& \simeq \Maps_C(c,y).\end{align*}  We have a homotopy
commutative diagram 
\begin{center}
\begin{tikzcd}[column sep = .2in]
\Maps_{C_*}(t \to c \sqcup t, t \to y)\arrow[d,"\operatorname{apply \,\,\, F}" left]\arrow[r,"\simeq"]& 
Z\arrow[d]\arrow[r,"\simeq"] &
 \Maps_C(c,y) \arrow[dl,"1\times (t \to y)"]
\arrow[ddl,bend left,"1"below]\\
\Maps_{C}( c \sqcup t, y)\arrow[r,"\simeq"] \arrow[d,"u_c \circ -" left] &\Maps_C(c,y) \times \Maps_C(t,y) \arrow[d, "\pi_1" left] \\
\Maps_C(c,y)\arrow[r,"="] & \Maps_C(c,y)
\end{tikzcd}
\end{center}
where $Z:=\op{fib}\left(\left(\Maps_C(c,y) \times \Maps_C(t,y)\right) \xrightarrow{\pi_2} \Maps_C(t,y)\right)$. The rightmost map is a weak homotopy equivalence, as are all the horizontal arrows, so the left vertical composite is too.
\end{proof}

\end{appendices}

\bibliographystyle{abbrv}
\bibliography{top_horrocks_ARXIV_2}

\end{document}